\documentclass{article}

\usepackage[arxiv,final]{optional}

\usepackage{graphicx}
\usepackage{amssymb}
\usepackage{amsmath}
\usepackage{amsthm}
\usepackage{paralist}
\usepackage[T1]{fontenc}
\usepackage{lmodern}

\opt{draft}{
\usepackage{refcheck}
\usepackage{todonotes}
}
\opt{final}{
\newcommand{\todo}[1]{}

}

\usepackage{algorithm}
\usepackage{algorithmic}

\opt{els}{
\usepackage{tikz}
\usepackage{pgfplots}
\pgfplotsset{compat=1.8}
}

\newcommand{\R}{\ensuremath{\mathbb R}}
\newcommand{\C}{\ensuremath{\mathbb C}}

\newcommand{\abs}[1]{\left|#1\right|}

\newcommand{\ones}{\ensuremath{ 1}}

\newcommand{\setcond}{\ensuremath{\mid}}

\DeclareMathOperator{\flatten}{vec}

\DeclareMathOperator{\image}{im}
\DeclareMathOperator{\lspan}{span}
\DeclareMathOperator{\diag}{diag}

\newcommand{\bigo}[1]{\ensuremath{\mathcal{O}\left(#1\right)}}

\newcommand{\defby}{\mathrel{\mathop:}=}

\newcommand{\bydef}{=\mathrel{\mathop:}}

\newcommand{\norm}[1]{\left\lVert #1 \right\rVert}

\newcommand{\twobyone}[2]{\left[\begin{smallmatrix} {#1}\\{#2} \end{smallmatrix}\right]}

\graphicspath{{figures/}}

\newtheorem{theorem}{Theorem}[section]
\newtheorem{lemma}[theorem]{Lemma}

\newtheorem{definition}[theorem]{Definition}
\newtheorem{remark}[theorem]{Remark}
\newtheorem{example}[theorem]{Example}

\opt{els}{
\journal{Linear Algebra and its Applications}
}

\newcommand{\titlestring}{Fast Recovery and Approximation of Hidden
Cauchy Structure}

\newcommand{\authorfnstringa}{TU Berlin, Institut f\"ur Mathematik, MA
4-5, Stra{\ss}e des 17. Juni 136, 10623 Berlin, Germany, {\ttfamily
liesen@math.tu-berlin.de}}

\newcommand{\authorfnstringb}{\'Ecole polytechnique f\'ed\'erale de
Lausanne, SB MATHICSE ANCHP, MA 2 B2 454, Station 8,
CH-1015 Lausanne, Switzerland, {\ttfamily
robert.luce@epfl.ch}}

\begin{document}

\opt{arxiv}{
\title{\titlestring}
\author{J\"org Liesen\thanks{\authorfnstringa} \and
Robert Luce\thanks{\authorfnstringb}}
\maketitle
}

\opt{els}{
\begin{frontmatter}

\title{\titlestring}

\author{J\"org Liesen\fnref{affilfna}}
\author{Robert Luce\fnref{affilfnb}}

\fntext[affilfna]{\authorfnstringa}
\fntext[affilfnb]{\authorfnstringb}

}

\begin{abstract}
We derive an algorithm of optimal complexity which determines whether
a given matrix is a Cauchy matrix, and which exactly recovers the
Cauchy points defining a Cauchy matrix from the matrix entries.
Moreover, we study how to approximate a given matrix by a Cauchy
matrix with a particular focus on the recovery of Cauchy points from
noisy data. We derive an approximation algorithm of optimal
complexity for this task, and prove approximation bounds.  Numerical
examples illustrate our theoretical results.
\end{abstract}

\opt{els}{%
\begin{keyword}
Cauchy matrix, difference matrix, data recovery, linearization, linear
least squares problem, approximation, pseudoinverse
\MSC[2010] 15B05, 65Y20, 65F20
\end{keyword}
\end{frontmatter}
}

\section{Introduction}

Two vectors $s \in \C^m$, $t \in \C^n$ are called \emph{Cauchy
points}, if
\begin{equation*}
    s_i - t_j \neq 0 \quad \text{for all $i,j$}.
\end{equation*}
Such Cauchy points define a
\emph{Cauchy matrix}
\begin{equation*}
    C(s,t) = [c_{ij}] \defby \left[ \frac{1}{s_i - t_j} \right].
\end{equation*}
Cauchy matrices occur in numerous applications. To give just one example,
let $(s_i,z_i)\in\C\times\C$ be given with pairwise distinct values $s_1,\dots,s_n$
and let $t_1,\dots,t_n\in \C$ be given with $s_i\neq t_j$ for all $i,j$. Then the
coefficients $a=[a_1,\dots,a_n]^T\in\C^n$ such that the rational function
\begin{equation*}
    r(\zeta) = \sum_{j=1}^n \frac{a_j}{\zeta - t_j}
\end{equation*}
satisfies $r(s_i)=z_i$, $i=1,\dots,n$, can be found by solving the
linear system
\begin{equation*}
    C(s,t) a = z.
\end{equation*}
Note that the condition $s_i- t_j\neq 0$ for the Cauchy points appears naturally in this
application (as in many others) by the requirement that the poles of the rational
function $r(\zeta)$ must be distinct from the points where the
(finite) values of $r(\zeta)$
are prescribed.

A Cauchy matrix satisfies the Sylvester type displacement equation
\begin{equation*}
    SC(s,t)-C(s,t)T=\ones_m\ones_n^T,
\end{equation*}
where $S \defby \diag(s) \in
\C^{m,m}$, $T \defby \diag(t) \in \C^{n,n}$, and $\ones_m\defby
[1,\dots,1]^T\in\R^m$. Hence the \emph{$\{S,T\}$-displacement rank} of
$C(s,t)$ is equal to $1$. The concept of displacement rank was
originally introduced in~\cite{FriedlanderEtAl1979,KailathKungMorf1979};
see~\cite[Section~12.1]{GolubVL2013} for an introduction.
Due to this special structure, several fast algorithms exist for
performing matrix computations with $C(s,t)$. For example, an $LU$
decomposition of $C(s,t)$ with partial pivoting can be computed in
$\bigo{mn}$ operations~\cite{GKO1995} (the GKO algorithm), and
matrix-vector products with $C(s,t)$ can be computed very
fast~\cite{GreengardRokhlin1987} (the fast multipole method); see
also~\cite{GohbergOlshevsky1994} and~\cite[Section~3.6]{Pan2001}.

In this work we are, however, not concerned with performing
computations with Cauchy matrices.  Rather we study the problem of
determining whether a given matrix $A \in \C^{m,n}$ is equal or at
least ``close'' to a Cauchy matrix. For such matrices we derive
algorithms of \textit{optimal complexity} that compute Cauchy points
$s \in \C^m$, $t \in \C^n$ with $A=C(s,t)$ when $A$ is a Cauchy
matrix, or with $A \approx C(s,t)$ when certain conditions are
satisfied. We are not aware that a similar study has appeared in the
literature before.

This cheap recognition (and approximation) could possibly be useful in
black-box linear system solvers: Instead of using a general purpose
method, one could first run the proposed algorithms in order to
determine whether the given matrix is close to a Cauchy matrix,
and then solve the system with a specialized algorithm. The upfront
test runs in time proportional to the size of the input, and hence
the computational overhead is negligible.

Let us briefly describe our general approach and the outline of this paper.
When $A = [a_{ij}] = C(s,t)$ is a Cauchy matrix, but the corresponding
Cauchy points $s,t$ are unknown, these can be computed by solving
the $mn$ nonlinear equations (in $m+n$ variables)
\begin{equation}
\label{eqn:nonlinear}
    \frac{1}{s_i - t_j} = a_{ij}, \quad 1 \le i \le m, \; 1 \le j \le n.
\end{equation}
For the Cauchy matrix $A$ we have $a_{ij} \neq 0$, and hence the equations
\eqref{eqn:nonlinear} are \emph{equivalent} to the $nm$ linear equations
(in $m+n$ variables)
\begin{equation}
\label{eqn:linear}
s_i - t_j = \frac{1}{a_{ij}},  \quad 1 \le i \le m, \; 1 \le j \le n.
\end{equation}
In Section~\ref{sec:exact} we discuss the
\emph{linearization}~\eqref{eqn:linear} of the equations
\eqref{eqn:nonlinear} in more detail, study uniqueness properties of
its solution and derive an algorithm for
solving~\eqref{eqn:linear} in $\bigo{m+n}$ operations.

If the given matrix $A=[a_{ij}]$ is not a Cauchy matrix, and the task is to
\emph{approximate} $A$ with a Cauchy matrix, one would ideally like
to solve the nonlinear optimization problem
\begin{equation}
\label{eqn:nonconvex}
    \min_{s, t} \sum_{i=1}^m \sum_{j=1}^n
        \abs{ \frac{1}{s_i - t_j} - a_{ij} }^2 \,
        = \, \min_{s, t} \norm{C(s,t)-A}_F^2.
\end{equation}
Instead of solving~\eqref{eqn:nonconvex}, we consider the linear least
squares problem
\begin{equation}
\label{eqn:convex}
    \min_{s, t} \sum_{i=1}^m \sum_{j=1}^n
        \abs{ s_i - t_j - \frac{1}{a_{ij}} }^2 \,
        =\, \min_{s, t} \|D(s,t)-A^{[-1]}\|_F^2,
\end{equation}
where
\begin{equation*}
    A^{[-1]}\defby [a_{ij}^{-1}], \quad D(s,t):=[s_i-t_j]\in\C^{m,n}.
\end{equation*}
The problem \eqref{eqn:convex} can be considered a
\textit{linearization} of the nonlinear problem \eqref{eqn:nonconvex}.
We first show in Section~\ref{sec:algorithm} how to
solve~\eqref{eqn:convex} in $\bigo{nm}$ operations. In
Section~\ref{sec:bounds} we relate the solutions obtained
from~\eqref{eqn:convex} to solutions of the original
problem~\eqref{eqn:nonconvex}.  In particular, we analyze when a
solution of~\eqref{eqn:convex} delivers a good approximation to the
Cauchy points of a ``noisy'' Cauchy matrix $A = C(s,t) + N$, where the
matrix $N$ represents some data error. We illustrate our results by
numerical experiments in Section~\ref{sec:examples}. 
Concluding remarks are given in Section~\ref{sec:conclusion}.

\paragraph*{Notation}

The vector (matrix) of all ones in $\C^n$ ($\C^{m,n}$) is denoted
by $\ones_n$ ($\ones_{m,n}$). For a matrix $A=[a_{ij}] \in \C^{m,n}$,
\begin{equation*}
    \norm{A}_F = \left({\sum_{i,j} \abs{a_{ij}}^2}\right)^{1/2}
    \quad \mbox{and} \quad
    \norm{A}_M = \max_{i,j} \abs{a_{ij}}
\end{equation*}
denote its Frobenius and maximum norm, respectively. Provided that all
the entries of $A$ are nonzero, its elementwise inverse is
$A^{[-1]}\defby[a_{ij}^{-1}]$, and $A^{[-T]}\defby (A^{[-1]})^T$.  For
two matrices $A,B$ of appropriate sizes we denote by $A \odot B$ and
$A \otimes B$ their Hadamard (elementwise) and Kronecker products,
respectively. Finally, $\flatten(A) \in \C^{mn}$ denotes the vector
resulting from stacking all the columns of $A \in \C^{m,n}$ upon
another.

\section{Exact recovery of Cauchy points}
\label{sec:exact}

Let $A=[a_{ij}]\in\C^{m,n}$ with $a_{ij}\neq 0$ for all $i,j$ be given.
There exist Cauchy points $s\in \C^m$, $t \in \C^n$ with $A=C(s,t)$,
i.e., $A$ is a Cauchy matrix, if and only if the equations \eqref{eqn:nonlinear}
hold. Since $a_{ij}\neq 0$ for all $i,j$, the equations \eqref{eqn:nonlinear}
are equivalent with the equations~\eqref{eqn:linear}, and these can be written
in matrix form as
\begin{equation}
\label{eqn:cauchy_system}
U\begin{bmatrix}
  s\\t
 \end{bmatrix}
\,=\,b,
\end{equation}
where
\begin{equation}
\label{eqn:matrixwise}
    U \defby
    \begin{bmatrix}
        I_m \otimes \ones_n & -\ones_m \otimes I_n
    \end{bmatrix}\in\C^{mn,(m+n)},
    \quad
    b \defby \flatten(A^{[-T]})\in\C^{mn}.
\end{equation}
%
Using the (overdetermined) linear system
\eqref{eqn:cauchy_system}--\eqref{eqn:matrixwise} we can test whether
a given matrix $A=[a_{ij}]$ with $a_{ij}\neq 0$ for all $i,j$ is a
Cauchy matrix or not:

If $\twobyone{s}{t}$ solves \eqref{eqn:cauchy_system}--\eqref{eqn:matrixwise}
for a componentwise nonzero right hand side $b$, then $s_i - t_j \neq 0$
for all $i,j$ (cf.~\eqref{eqn:linear}), so that $s,t$ are Cauchy points and
$A = C(s,t)$. On the other hand, there are, of course, matrices $A$ with all
entries nonzero, giving a componentwise nonzero $b$, for which no
solution of \eqref{eqn:cauchy_system}--\eqref{eqn:matrixwise} exists.

\begin{example}
For $A=\left[\begin{smallmatrix}1 & -1\\-1 & 1\end{smallmatrix}\right]$ we have
\begin{equation*}
U=
\begin{bmatrix}
    1 & 0 & -1 & 0\\
    1 & 0 &  0 & -1\\
    0 & 1 & -1 & 0\\
    0 & 1 &  0 & -1
\end{bmatrix},
\quad
b = \flatten(A^{[-T]}) =
\begin{bmatrix}
   1 \\ -1 \\ -1 \\1
\end{bmatrix},
\end{equation*}
and a simple computation shows that there exists no solution of $U
\twobyone{s}{t} = b$.  Hence $A$ is not a Cauchy matrix.
\end{example}

If $s\in\C^m$, $t\in\C^n$ are Cauchy points, then
\begin{equation*}
    C(s,t) = C(s + \alpha \ones_m, t + \alpha \ones_n)
\end{equation*}
for all $\alpha \in \C$.  Consequently, the Cauchy points $s,t$ of a
Cauchy matrix $A = [a_{ij}]$ are not uniquely determined by the values
$a_{ij}$.  We will show next that this global translation of the
Cauchy points is the only source of ambiguity.

\begin{theorem}
\label{thm:uniqueness}
The matrix $U$ in~\eqref{eqn:matrixwise} satisfies $\ker(U) = \lspan\{ \ones_{m+n} \}$.
Thus, if $\left[\begin{smallmatrix} s\\t\end{smallmatrix}\right]$ is a solution
of \eqref{eqn:cauchy_system}--\eqref{eqn:matrixwise}, then the set of all solutions
is given by
\begin{equation*}
    \left\{\left[ \begin{smallmatrix} s\\t \end{smallmatrix} \right]
        + \alpha \ones_{m+n} \setcond \alpha \in \C \right\}.
\end{equation*}
\end{theorem}
\begin{proof}
Since $U\ones_{m+n}=0$ we have  $\lspan\{ \ones_{m+n} \}\subseteq \ker(U)$.
If $z=\left[\begin{smallmatrix} x\\y\end{smallmatrix}\right]\in \ker(U)$ with
$x\in\C^m$ and $y\in\C^n$, then
\begin{equation*}
    x_j \ones_n = y, \quad j=1,\dotsc,m.
\end{equation*}
In particular, $y=x_1\ones_n$, which implies $x_j=x_1$ for $j=2,\dots,m$,
so that $\twobyone{x}{y} = x_1 \ones_{m+n}$, giving that $ \ker(U)
\subseteq \lspan\{ \ones_{m+n} \}$.
\end{proof}

In order to remove the ambiguity about the possible Cauchy points that
define a given Cauchy matrix we introduce the following definition.

\begin{definition}
\label{def:normalized_cp}
Let $A \in \C^{m,n}$ be a Cauchy matrix.  We say that $\tilde{s} \in
\C^m$, $\tilde{t} \in \C^n$ are \emph{normalized Cauchy points for
$A$}, if $A=C(\tilde{s},\tilde{t})$ and
$\norm{\twobyone{\tilde{s}}{\tilde{t}}}_2$
is minimal among all possible Cauchy points
$s\in\C^m$, $t\in\C^n$ with $A=C(s,t)$.
\end{definition}

If $A=C(s,t)$, then normalized Cauchy points
for $A$ can be found by solving the minimization problem
\begin{equation*}
    \min_{\alpha \in \C} \norm{\twobyone{s}{t} - \alpha
\ones_{m+n}}_2^2.
\end{equation*}
The unique solution is given by
\begin{equation*}
    \alpha_* \defby
    \frac{\ones_{m+n}^T \twobyone{s}{t}}{\ones_{m+n}^T\ones_{m+n}}
    = \frac{\ones_{m+n}^T \twobyone{s}{t}}{m+n},
\end{equation*}
and hence $\tilde{s},\tilde{t}$ with
\begin{equation*}
\begin{bmatrix}
    \tilde{s}\\
    \tilde{t}
\end{bmatrix}
\defby
\begin{bmatrix}
    s\\
    t
\end{bmatrix}
-\alpha_*\ones_{m+n}
\end{equation*}
are normalized Cauchy points for $A$.

\begin{algorithm}[t]
\caption{Optimal recovery of normalized Cauchy points}
\label{alg:exact}
\begin{algorithmic}[1]
\REQUIRE Cauchy matrix $A=[a_{ij}] \in \C^{m,n}$.
(Thus, $a_{ij} \neq 0$ for all $i,j$.)
\ENSURE Normalized Cauchy points $\tilde{s}, \tilde{t}$ such
that $A = C(\tilde{s},\tilde{t})$.
\medskip
\STATE $s(1) \leftarrow 0$
    \quad \COMMENT{Choice arbitrary}
\STATE $t(1:n) \leftarrow s(1) - A(1,1:n)^{[-1]}$
\STATE $s(2:m) \leftarrow t(1) + A(2:m,1)^{[-1]}$
\STATE $\alpha_* \leftarrow \tfrac{1}{m+n} (\sum s_i + \sum t_j)$
\STATE $\tilde{s} \leftarrow s - \alpha_* \ones_m$
\STATE $\tilde{t} \leftarrow t - \alpha_* \ones_n$
\end{algorithmic}
\end{algorithm}

As described above, if $A$ is a Cauchy matrix, then Cauchy
points for $A$ can be computed by solving the system
\eqref{eqn:cauchy_system}--\eqref{eqn:matrixwise}.  Since the matrix
$U$ has rank $m+n-1$ (cf. Theorem~\ref{thm:uniqueness}), the points
can be computed by solving any full-rank subsystem of
\eqref{eqn:cauchy_system}--\eqref{eqn:matrixwise} with $m+n-1$ rows.
Due to the simple structure of $U$, the solution of this subsystem can
be computed in $\bigo{m+n}$ operations. One possible algorithm is
shown in Algorithm~\ref{alg:exact}. At the end of the algorithm we
normalize the computed Cauchy points (according to
Definition~\ref{def:normalized_cp}), which can be achieved in
$\bigo{m+n}$ operations as well. Note that only the first row and
column of $A$ are accessed by the algorithm.

If we do not know whether $A$ is a Cauchy matrix, we can still apply
Algorithm~\ref{alg:exact} to $A$. Since the algorithm only considers the
first row and column of $A$, it then costs (at most) $mn$ operations to check
whether indeed $A=C(\tilde{s},\tilde{t})$.

We summarize these observations in the following result.

\begin{theorem}
If $A \in \C^{m,n}$ is a Cauchy matrix, then Algorithm~\ref{alg:exact} yields
normalized Cauchy points $\tilde{s} \in \C^m$, $\tilde{t} \in \C^n$ with
$A=C(\tilde{s},\tilde{t})$ in $\bigo{m+n}$ operations. Moreover, for any matrix
$A \in \C^{m,n}$ it can be decided in $\bigo{mn}$ operations whether $A$ is a
Cauchy matrix.
\end{theorem}

Note that neither the recovery of Cauchy points, nor recognizing Cauchy structure
can be achieved asymptotically faster than stated in this theorem.

\section{Approximation with Cauchy matrices}

In order to (best) approximate a given matrix $A \in \C^{m,n}$ (having
only nonzero entries) by a Cauchy matrix, we would ideally like to solve
the nonlinear optimization problem \eqref{eqn:nonconvex}.  As
described in the Introduction, we will instead solve the
\textit{linearization} of this problem given by \eqref{eqn:convex}.
Using the notation of Section~\ref{sec:exact}, this standard linear
least squares problem can be equivalently written as (cf.
\eqref{eqn:cauchy_system}--\eqref{eqn:matrixwise})
\begin{equation}
\label{eqn:lin_vector}
    \min_{ s, t} \norm{ U \twobyone{s}{t} - b}_2^2.
\end{equation}

Algorithm~\ref{alg:exact} from Section~\ref{sec:exact} is clearly
inappropriate in this context, as there is no guarantee that the
submatrix of $U$ picked for the reconstruction of the Cauchy points
yields any useful \textit{global} approximation of the given data when
$A$ is not a Cauchy matrix. Our main goal in
Section~\ref{sec:algorithm} is to derive an algorithm of optimal
complexity $\bigo{mn}$ for solving~\eqref{eqn:lin_vector}.  In
Section~\ref{sec:bounds} we relate the (optimal) solution obtained by
this algorithm to the original problem~\eqref{eqn:nonconvex}.

\subsection{Fast solution of the least squares problem}
\label{sec:algorithm}

We will solve the least squares problem \eqref{eqn:lin_vector} using the
singular value decomposition of the matrix $U$. We have already characterized
the kernel of $U$ in Theorem~\ref{thm:uniqueness}.  The following result
gives a complete characterization of the nonzero singular values and
corresponding singular vectors.

\begin{lemma}
\label{lem:sv_char}
The nonzero singular values of the matrix $U$ in \eqref{eqn:matrixwise} are
\begin{align*}
    \sqrt{m+n} & \quad \text{(of multiplicity one)},\\
    \sqrt{m}   & \quad \text{(of multiplicity $n-1$),}\\
    \sqrt{n}   & \quad \text{(of multiplicity $m-1$)}.
\end{align*}
Moreover, the corresponding right singular vectors can be characterized
as
\begin{align*}
\sqrt{m+n}: \quad
& \lspan\left\{
    \begin{bmatrix}
        \sqrt{\tfrac{n}{m}} \ones_{m}\\
        -\sqrt{\tfrac{m}{n}} \ones_n
    \end{bmatrix}
\right\},\\
\sqrt{m}: \quad
& \lspan\left\{
    \begin{bmatrix}
        0_m\\ v
    \end{bmatrix}
    \setcond
    v \in \C^n, \ones_n^Tv = 0
\right\},\\
\sqrt{n}: \quad
& \lspan\left\{
    \begin{bmatrix}
        v\\ 0_n
    \end{bmatrix}
    \setcond
    v \in \C^m, \ones_m^Tv = 0
\right\},
\intertext{and the corresponding left singular vectors can be characterized as}
\sqrt{m+n}: \quad & \lspan\left\{ \ones_{mn} \right\},\\
\sqrt{m}: \quad & \lspan \left\{ \ones_m \otimes v \setcond v \in
\C^n, \ones_n^T v = 0 \right\},\\
\sqrt{n}: \quad & \lspan \left\{ v \otimes \ones_n \setcond v \in
\C^m, \ones_m^T v = 0 \right\}.
\end{align*}
\end{lemma}

\begin{proof}
The claims can be verified by straightforward computations using the matrix
\begin{equation*}
    U^TU =
    \begin{bmatrix}
        n I_m & - \ones_{m,n}\\
        - \ones_{n,m} & m I_n
    \end{bmatrix}
\end{equation*}
for the right singular vectors, and the matrix
\begin{equation*}
    UU^T = I_m\otimes \ones_{n,n} + \ones_{m,m} \otimes I_n
\end{equation*}
for the left singular vectors.
\end{proof}

The next theorem gives an explicit formula for the solution
of~\eqref{eqn:lin_vector}, which in particular shows that
this solution can be computed fast. We denote the Moore-Penrose
pseudoinverse of $U$ by $U^+$.

\begin{theorem}
\label{thm:fast_ls_sol}
Let $A \in \C^{m,n}$ have only nonzero entries. Let
$b \defby \flatten(A^{[-T]})$ and
\begin{equation*}
    r \defby \frac{1}{n} A^{[-1]} \ones_n \in \C^m, \quad
    c \defby \frac{1}{m} A^{[-T]} \ones_m \in \C^n, \quad
    \sigma \defby \frac{1}{mn} \ones_m^T A^{[-1]} \ones_n.
\end{equation*}
Then the minumum norm solution of $\min_{s,t} \norm{U\twobyone{s}{t} -
b}_2$ has the form
\begin{equation}\label{eqn:Upb}
    U^+ b
    =
    \begin{bmatrix}
        r -   \tfrac{m \sigma}{m+n} \ones_m\\
        - c + \frac{n \sigma}{m+n}  \ones_n
    \end{bmatrix},
\end{equation}
which can be computed in $\bigo{mn}$ operations.
Moreover, $U^+b$ yields Cauchy points if and only if
\begin{equation}
\label{eqn:UUplus}
    U U^+b \neq 0 \quad \text{(componentwise)},
\end{equation}
or, equivalently,
\begin{equation}
\label{eqn:means_cond}
    r_i + c_j \neq \sigma \quad \text{for all $i,j$}.
\end{equation}
\end{theorem}

\begin{proof}
For an integer $k \ge 1$, we denote by $Q_k \in \R^{k,k-1}$ a matrix
whose columns form an orthogonal basis for the linear subspace $\{
v \in \C^k \setcond \ones_k^T v = 0 \}$, so that $Q_k^T Q_k = I_{k-1}$
and $\ones_k^T Q_k = 0$.  The characterization of the singular
values of $U$ in Lemma~\ref{lem:sv_char} shows that
\begin{equation}\label{eqn:matrixW}
W =
\begin{bmatrix}
    \sqrt{\frac{n}{m(m+n)}} \ones_m & 0_{m, n-1} & Q_m &
        \frac{1}{\sqrt{m+n}} \ones_m\\
    - \sqrt{\frac{m}{n(m+n)}} \ones_n & Q_n & 0_{n, m-1} &
        \frac{1}{\sqrt{m+n}} \ones_n\\
\end{bmatrix} \in \R^{m+n, m+n}
\end{equation}
is orthogonal, and yields a diagonalization $U^T U = W \Lambda W^T$ with
\begin{equation*}
\Lambda =
    \diag(m+n, \underbrace{m, \dotsc, m}_{n-1},
    \underbrace{n, \dotsc, n}_{m-1}, 0),
\end{equation*}
so that
\begin{equation}
\label{eqn:MPformula}
    U^+ = W \Lambda^+ W^T U^T.
\end{equation}
Since $\ones_m^T Q_m = 0$, the matrix $\hat{Q} = [Q_m,
m^{-\frac{1}{2}} \ones_m]$ is orthogonal and hence $I_m = \hat{Q} \hat{Q}^T
= Q_m Q_m^T + \frac{1}{m} \ones_{m,m}$, which implies that
\begin{equation}
\label{eqn:outprod_Q}
Q_m Q_m^T = I_m - \tfrac{1}{m} \ones_{m,m}.
\end{equation}
%
%
%
Noting that $U^T b = \twobyone{n r}{-m c}$, we compute
from~\eqref{eqn:MPformula}, using~\eqref{eqn:outprod_Q},
\begin{equation*}
\begin{split}
    U^+ b
& = W \Lambda^+ W^T
    \begin{bmatrix}
        nr\\-mc
    \end{bmatrix}
    = W \Lambda^+
    \begin{bmatrix}
        \sqrt{mn(m+n)} \sigma\\
        - m Q_n^T c\\
        n Q_m^T r\\
        0
    \end{bmatrix}
    = W
    \begin{bmatrix}
        \frac{\sqrt{mn}\sigma}{\sqrt{m+n}}\\
        - Q_n^T c\\
        Q_m^T r\\
        0
    \end{bmatrix}\\
& =
    \begin{bmatrix}
         r - \tfrac{m\sigma}{m+n} \ones_m\\
        - c + \frac{n\sigma}{m+n} \ones_n
    \end{bmatrix}.
\end{split}
\end{equation*}
Evaluating the last expression for $U^+b$ takes $\bigo{mn}$
operations.

Finally, with $\twobyone{s}{t}\defby U^+b$ the condition \eqref{eqn:UUplus} 
simply means that $s_i-t_j\neq 0$ for all $i,j$, or, equivalently,
\begin{equation*}
    \left(  r_i - \tfrac{m\sigma}{m+n} \right) -
    \left( -c_j + \tfrac{n\sigma}{m+n} \right) =
    r_i + c_j- \sigma \neq 0
\end{equation*}
for all $i,j$.
\end{proof}

Note that $r$ and $c$ in Theorem~\ref{thm:fast_ls_sol} are the vectors of row and
column means of the matrix $A^{[-1]}$, respectively, while $\sigma$ is the mean of all its
entries. Moreover, for a Cauchy matrix $A=C(s,t)$ the condition \eqref{eqn:means_cond}
reduces to $s_i-t_j\neq 0$ for all $i,j$.

The overall algorithm for computing $U^+b$ according to Theorem~\ref{thm:fast_ls_sol}
is shown in Algorithm~\ref{alg:global}.

\begin{remark}
An explicitly constructed matrix $Q_m$ satisfying the requirements in
the proof of Theorem~\ref{thm:fast_ls_sol} is given
in~\ref{sec:explicit_svd}.  Consequently, a singular value
decomposition of the matrix $U$, based on Lemma~\ref{lem:sv_char},
can be constructed explicitly.
\end{remark}

\begin{algorithm}[t]
\caption{Minimum 2-norm solution of the least squares problem~\eqref{eqn:lin_vector}.
\label{alg:global}}
\begin{algorithmic}[1]
\REQUIRE Matrix $A=[a_{ij}] \in \C^{m,n}$ with $a_{ij} \neq 0$ for all $i,j$.
\ENSURE $\twobyone{s}{t}=U^+\flatten(A^{[-T]})$.
\medskip
\STATE $r      \leftarrow \frac{1}{n}  A^{[-1]} \ones_n$
\STATE $c      \leftarrow \frac{1}{m}  A^{[-T]} \ones_m$
\STATE $\sigma \leftarrow \frac{1}{mn} \ones_m^T A^{[-1]} \ones_n$
\STATE $s \leftarrow r - \frac{m \sigma}{m+n} \ones_m$
\STATE $t \leftarrow \frac{n\sigma}{m+n} \ones_n - c$
$\quad$ \COMMENT{min. 2-norm solution is automatically normalized}
\end{algorithmic}
\end{algorithm}

The following example gives a matrix $A$ with only nonzero entries
for which Algorithm~\ref{alg:global} does not yield Cauchy points.

\begin{example}
Let $0 \neq \alpha \in \C$ and consider the matrix
\begin{equation*}
A =
\begin{bmatrix}
    \frac{1}{\alpha}       & - \frac{1}{\alpha + 2}\\
    - \frac{1}{\alpha - 2} & \frac{1}{\alpha}
\end{bmatrix}\quad\mbox{so that}\quad
A^{[-T]} =
\begin{bmatrix}
    \alpha       & - (\alpha - 2)\\
    - (\alpha + 2) & \alpha
\end{bmatrix},
\end{equation*}
which gives $r = \twobyone{-1}{1}$, $c = \twobyone{-1}{1}$, and $\sigma=0$.
The condition \eqref{eqn:means_cond} does not hold, so that $U^+b$ 
in \eqref{eqn:Upb}, or the output of Algorithm~\ref{alg:global}
applied to $A$, does not give Cauchy points.
\end{example}

In the next section we will derive conditions under which the output
of Algorithm~\ref{alg:global} results in good approximations to
the original problem~\eqref{eqn:nonconvex}.

\subsection{Approximation bounds}
\label{sec:bounds}

For each matrix $A\in \C^{m,n}$ with only nonzero entries a minimum $2$-norm solution
$\hat{z}=\left[\begin{smallmatrix}\hat{s}\\
\hat{t}\end{smallmatrix}\right]$ of the least squares
problem~\eqref{eqn:lin_vector} and hence of~\eqref{eqn:convex} can be
computed in $\bigo{mn}$ operations using Algorithm~\ref{alg:global}.
Of course, without further assumptions we cannot expect that $\hat{z}$
closely approximates the solution of the nonlinear
problem~\eqref{eqn:nonconvex}.
Below we will derive a bound on
$\|A-C(\hat{s},\hat{t})\|_F$, and we will bound
$\|\left[\begin{smallmatrix}s\\ t\end{smallmatrix}\right]-
\left[\begin{smallmatrix}\hat{s}\\
\hat{t}\end{smallmatrix}\right]\|_2$ for a perturbed Cauchy matrix
$A=C(s,t)+N$. In our derivations we will use that the
Hadamard product is submultiplicative with respect to the Frobenius norm, i.e.,
\begin{equation*}
    \norm{ A \odot B }_F \leq \norm{A}_F \norm{B}_F;
\end{equation*}
see, e.g.,~\cite[equation (3.3.5)]{Hor1990}.

Our first result connects the residuals of~\eqref{eqn:nonconvex}
and~\eqref{eqn:convex}.  It shows that if for given vectors ${s}, {t}$
the relative residual of the linearization~\eqref{eqn:convex} is
reasonably small, then ${s}, {t}$ are Cauchy points,
\emph{and} their relative error with respect to the original
problem~\eqref{eqn:nonconvex} is small as well.  Note that the theorem
applies in particular to the output of Algorithm~\ref{alg:global},
since it computes an \emph{optimal} solution for
the linearization~\eqref{eqn:convex}.
Recall that $D({s},{t}) = [{s}_i - {t}_j] \in \C^{m,n}$.

\begin{theorem}
\label{thm:data_res}
Let $A\in\C^{m,n}$ have only nonzero entries and let ${s}\in\C^m$, ${t}\in\C^n$. 
Define the residual matrix corresponding to~\eqref{eqn:convex} by $R\defby
A^{[-1]}-D({s},{t})$.
If
\begin{equation}\label{eqn:res_condition}
    \norm{A \odot R}_M \bydef \beta < 1,
\end{equation}
then
\begin{equation*}
    \min_{i,j} |{s}_i-{t}_j| \ge \norm{A}_M^{-1} (1 - \beta),
\end{equation*}
and hence, in particular, ${s},{t}$ are Cauchy points. Moreover,
\begin{equation}
\label{eqn:beta_bound}
    \frac{\norm{A - C({s}, {t})}_F}{\norm{A}_F}
        \le \frac{\beta}{1 - \beta}.
\end{equation}
\end{theorem}

\begin{proof}
Let $R=[r_{ij}]$, then for all $i,j$ we get
\begin{equation*}
    \abs{{s}_i - {t}_j}
    = \abs{ \frac{1}{a_{ij}} - r_{ij} }
    = \frac{ \abs{ 1 - a_{ij} r_{ij} } }{ \abs{a_{ij}} }
    \ge \norm{A}_M^{-1} (1 - \beta),
\end{equation*}
which shows the lower bound on $\min_{i,j} |{s}_i-{t}_j|$.

In order to prove \eqref{eqn:beta_bound} we compute
\begin{equation*}
a_{ij} - \frac{ 1 }{ {s}_i - {t}_j }
= a_{ij} - \frac{ 1 }{ \frac{1}{a_{ij}} - r_{ij} }
= a_{ij} \left( 1 - \frac{1}{1 - a_{ij} r_{ij}}  \right)
= a_{ij} \frac{a_{ij} r_{ij}}{1 - a_{ij} r_{ij}},
\end{equation*}
so that
\begin{equation*}
\left|a_{ij} - \frac{ 1 }{ {s}_i - {t}_j }\right|\leq |a_{ij}|\frac{\beta}{1-\beta},
\end{equation*}
giving $\|A-C({s},{t})\|_F\leq \frac{\beta}{1-\beta}\|A\|_F$.
\end{proof}

The condition \eqref{eqn:res_condition} can be written as
\begin{equation}
\label{eqn:maxcompres}
\max_{i,j}
\left|
    \frac{({s}_i-{t}_j)-a_{ij}^{-1}}{a_{ij}^{-1}}
\right|
= \beta < 1.
\end{equation}
In words, the maximal compentwise relative error in the linear
equations~\eqref{eqn:linear} that is made by the vectors
${s},{t}$ has to be smaller than one. This appears to be a
natural and in fact minimal assumption on the output of
Algorithm~\ref{alg:global} so that it gives any useful information
about the optimization problems~\eqref{eqn:nonconvex}
and~\eqref{eqn:convex}.  This maximal componentwise relative error can
be larger than the global relative error $\norm{D({s},
{t}) - A^{[-1]}}_F / \norm{A^{[-1]}}_F$, especially if the entries
of $A$ vary greatly in magnitude.  In that case the
bound~\eqref{eqn:beta_bound} (and the approximation error) is
adversely affected; see Section~\ref{sec:examples} for an example.

In the next result we investigate how closely the output of
Algorithm~\ref{alg:global} approximates the Cauchy points of a
perturbed Cauchy matrix $A$.

\begin{theorem}
\label{thm:cauchy_err}
Let $A = C(\tilde{s},\tilde{t}) + N\in\C^{m,n}$, where
$\tilde{s},\tilde{t}$ are normalized Cauchy points, have only nonzero
entries. Let ${s}\in\C^m,{t}\in\C^n$ be a minimum 2-norm solution of the
least squares problem \eqref{eqn:lin_vector}, i.e., the output of
Algorithm~\ref{alg:global} applied to $A$.  If
\begin{equation}
\label{eqn:err_condition}
    \norm{D(\tilde{s},\tilde{t}) \odot N}_M \bydef \gamma < 1,
\end{equation}
then
\begin{equation}
\label{eqn:gamma_bound}
    \frac{\norm{\left[
        \begin{smallmatrix}
            \tilde{s}\\
            \tilde{t}
        \end{smallmatrix}
    \right]
    - \left[
        \begin{smallmatrix}
            {s}\\
            {t}
        \end{smallmatrix}
    \right]}_2}{\norm{\left[
        \begin{smallmatrix}
            \tilde{s}\\
            \tilde{t}
        \end{smallmatrix}\right]}_2}
    \le \frac{\sqrt{m+n}}{\min \{ \sqrt{m}, \sqrt{n} \}}
        \frac{\gamma}{1 - \gamma}.
\end{equation}
\end{theorem}
\begin{proof}
Let us denote $C=C(\tilde{s},\tilde{t})$, $N=[n_{ij}]$ and define
\begin{equation*}
B=[b_{ij}]\defby A^{[-T]}-C^{[-T]}.
\end{equation*}
Since ${s},{t}$ is a minimum 2-norm least squares solution
(cf. Theorem~\ref{thm:fast_ls_sol}), we have
\begin{equation*}
    \begin{bmatrix}
        {s}\\
        {t}
    \end{bmatrix}
    = U^+ \flatten(A^{[-T]})
    = U^+ ( \flatten( C^{[-T]} ) + \flatten(B ) )
    =
    \begin{bmatrix}
        \tilde{s}\\
        \tilde{t}
    \end{bmatrix}
    + U^+ \flatten(B).
\end{equation*}
We thus get
\begin{equation*}
\begin{split}
\norm{\begin{bmatrix}\tilde{s}\\ \tilde{t}\end{bmatrix}
    - \begin{bmatrix}{s}\\ {t}\end{bmatrix}}_2
    & = \norm{ U^+ \flatten(B) }_2
    \le \norm{ U^+ }_2 \norm{B}_F\\
    & = \frac{1}{ \min \{ \sqrt{m}, \sqrt{n} \} } \norm{B}_F,
\end{split}
\end{equation*}
where we used Lemma~\ref{lem:sv_char} in the last step.

It remains to bound $\norm{B}_F$.  Note first that for all $i,j$ we have
\begin{equation*}
\begin{split}
    |b_{ji}|
    & = \left|\left( \frac{1}{\tilde{s}_i - \tilde{t}_j} + n_{ij} \right)^{-1}
        - \left( \frac{1}{\tilde{s}_i - \tilde{t}_j} \right)^{-1}\right|
        = \left|(\tilde{s}_i - \tilde{t}_j) \frac{ (\tilde{s}_i - \tilde{t}_j) n_{ij}}{1 + (\tilde{s}_i - \tilde{t}_j) n_{ij}}\right|\\
     & \leq \abs{(\tilde{s}_i - \tilde{t}_j)}
        \frac{\gamma}{1 - \gamma},
\end{split}
\end{equation*}
resulting in
\begin{equation*}
\begin{split}
    \norm{B}_F & \leq \norm{ D(\tilde{s}, \tilde{t}) }_F
        \frac{ \gamma }{ 1 - \gamma }
    = \norm{ U \twobyone{\tilde{s}}{\tilde{t}} }_2
        \frac{ \gamma }{ 1-\gamma }
    \leq \norm{U}_2 \norm{ \twobyone{\tilde{s}}{\tilde{t}} }_2
        \frac{ \gamma }{ 1 - \gamma }\\
    & = \sqrt{m+n}\norm{ \twobyone{\tilde{s}}{\tilde{t}} }_2
        \frac{ \gamma }{ 1 - \gamma },
\end{split}
\end{equation*}
where we have again used Lemma~\ref{lem:sv_char} in the last step.

\end{proof}

The condition \eqref{eqn:err_condition}, i.e.,
\begin{equation*}
    \norm{[(\tilde{s}_i-\tilde{t}_j)n_{ij}]}_M
    = \norm{\left[
        \frac{n_{ij}}{\frac{1}{\tilde{s}_i - \tilde{t}_j}}
    \right]}_M
    < 1,
\end{equation*}
ensures that the maximum (compement wise) relative noise level is
reasonably small. Note also that the constant on the right hand side
of~\eqref{eqn:gamma_bound} is equal to $\sqrt{2}$ when $m=n$.

The two bounds presented in Theorems~\ref{thm:data_res}
and~\ref{thm:cauchy_err} are complementary: On the one hand, 
a small residual~\eqref{eqn:beta_bound} \emph{does not} imply that
Algorithm~\ref{alg:global} recovers nearby Cauchy points of a noisy
Cauchy matrix as in~\eqref{eqn:gamma_bound}. On the other hand,
if Algorithm~\ref{alg:global} recovers nearby Cauchy points of a noisy
Cauchy matrix as in~\eqref{eqn:gamma_bound}, then this \emph{does not}
imply that the residual~\eqref{eqn:beta_bound} is small. Numerical
examples demonstrating this are given in in Section~\ref{sec:examples}.

\begin{remark}
Without further assumptions on $N$ it is not guaranteed that the
output of Algorithm~\ref{alg:global} applied to a noisy Cauchy matrix $A
= C(\tilde{s},\tilde{t}) + N$ (with only nonzero entries)
yields Cauchy points.  However, considering~\eqref{eqn:UUplus}, 
the output $U^+ \flatten(A^{[-T]})$ are indeed Cauchy points 
if $\norm{N}$ is sufficiently small, since the
function
\begin{equation*}
    \{A \in \C^{m,n} \setcond a_{ij} \neq 0\} \rightarrow \C^{mn},
    \quad
    A \mapsto U U^+ \flatten(A^{[-T]}),
\end{equation*}
is continuous, and $U U^+ \flatten(C(\tilde{s},\tilde{t})^{[-T]}) \neq 0$
(componentwise).  We did not attempt to derive a quantitative bound on
$N$ such that $U^+ \flatten(A^{[-T]})$ are guaranteed to be Cauchy
points; see, however, the conditions \eqref{eqn:means_cond}
and \eqref{eqn:res_condition}.
\end{remark}

\subsection{Numerical examples}
\label{sec:examples}

\subsubsection*{Approximation quality of Algorithm~\ref{alg:global}}

We consider the vectors $s \in \C^{200}$ $t \in \C^{100}$, where the
real part consists of equally spaced points in the interval $[-1,1]$,
and imaginary parts set to  $i$ and $-i$, respectively, i.e.,
\begin{equation}
\label{eqn:points_ok}
    s = \text{{\ttfamily linspace}}(-1,1,200) + i, \quad 
    t = \text{{\ttfamily linspace}}(-1,1,100) - i
\end{equation}
in MATLAB syntax.  Consequently, the all the entries of the Cauchy
matrix $C \defby C(s,t)$ have the same magnitude.

\begin{figure}[p]
\begin{center}
    \includegraphics[width=0.7\textwidth]{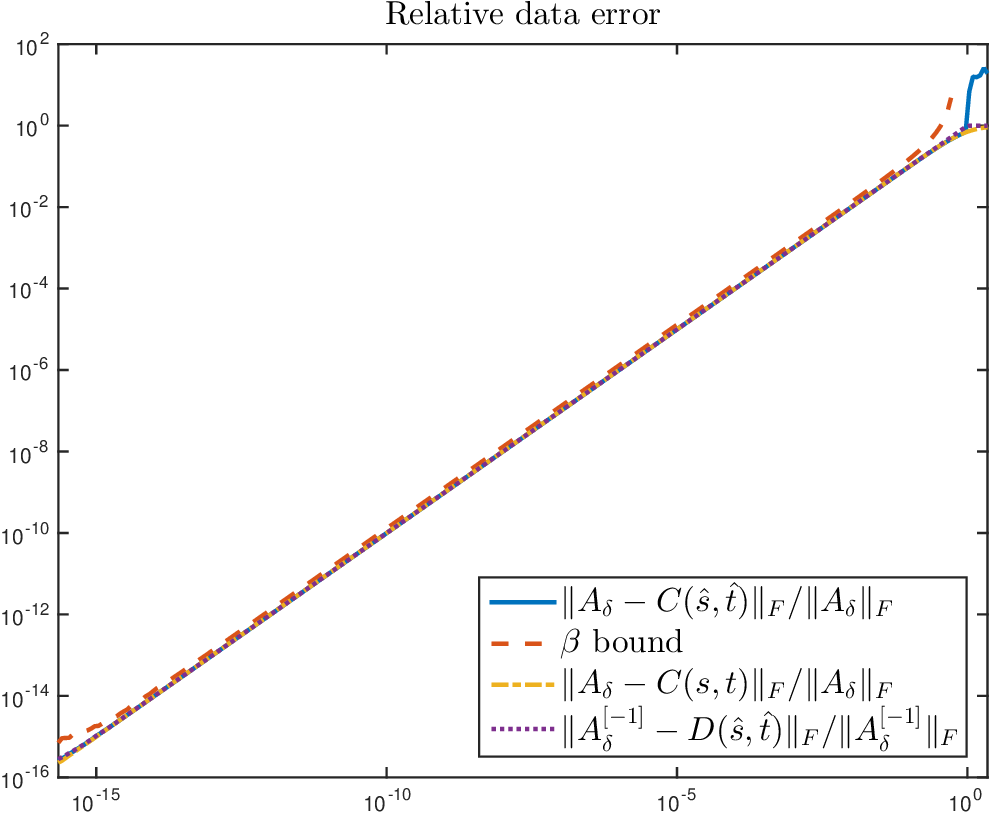}
    \includegraphics[width=0.7\textwidth]{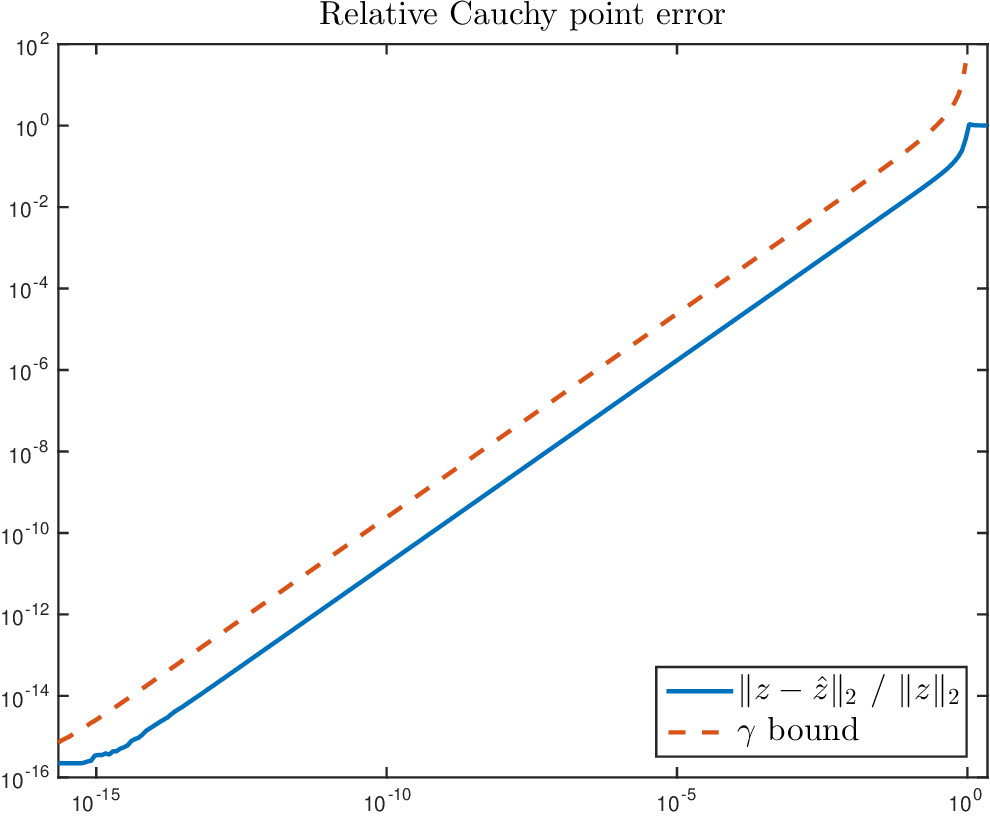}
\end{center}
\caption{Approximation quality of Algorithm~2 for the
data~\eqref{eqn:points_ok}--\eqref{eqn:points_noise}.  {\itshape{Top
picture:}} Relative data approximation error corresponding to the
Cauchy points obtained by Algorithm~\ref{alg:global} (solid blue
line), the original Cauchy points (dash-dotted yellow line), the
bound~\eqref{eqn:beta_bound} (dashed red line), and the relative
residual of the linearized problem~\eqref{eqn:convex} (dotted purple
line).  All four lines are visually almost indistinguishable.
{\itshape Bottom picture:}  Relative Cauchy points approximation error
corresponding to the Cauchy points obtained by
Algorithm~\ref{alg:global} (solid blue line), and the
bound~\eqref{eqn:gamma_bound} (dashed red line).
\label{fig:twolines_ok}}
\end{figure}

\begin{figure}[p]
\begin{center}
    \includegraphics[width=0.7\textwidth]{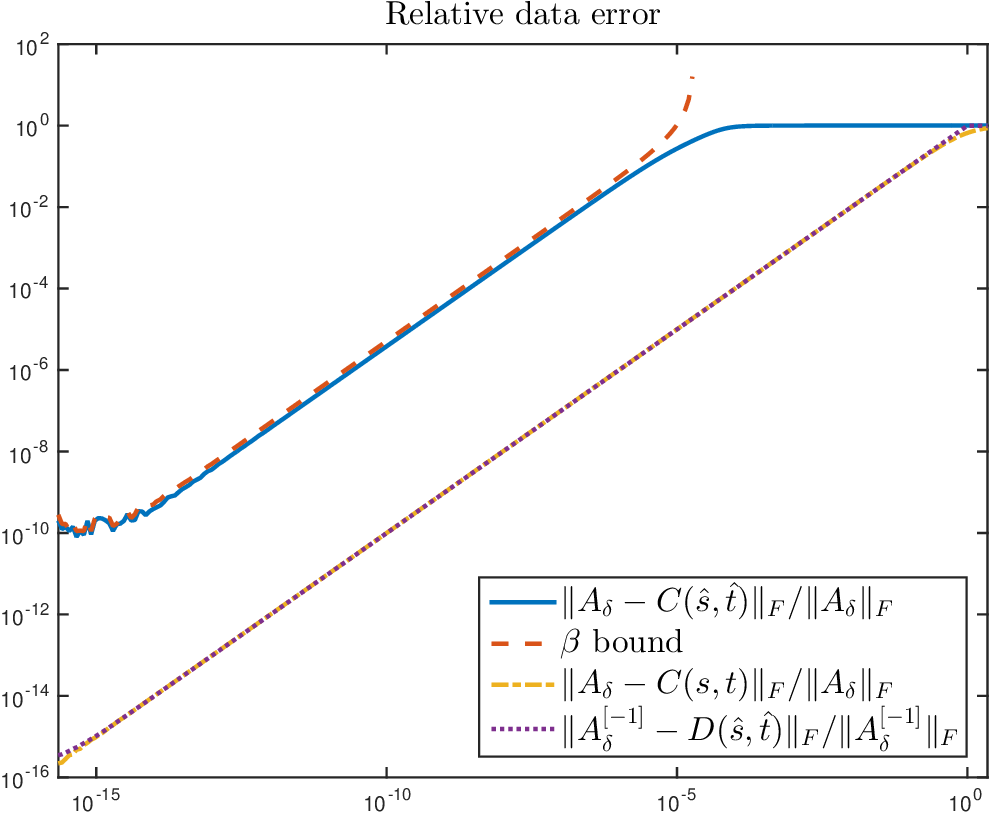}
    \includegraphics[width=0.7\textwidth]{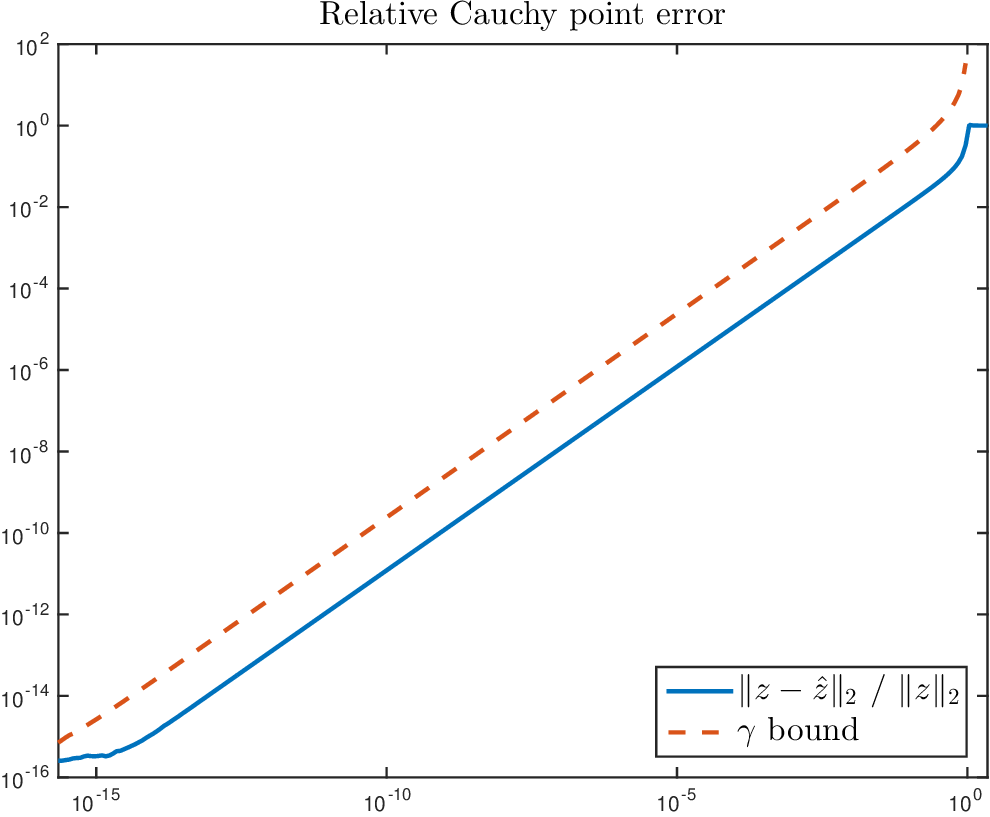}
\end{center}
\caption{Approximation quality of Algorithm~2 for the
data~\eqref{eqn:points_noise}--\eqref{eqn:points_bad}.  
Notation as in Figure~\ref{fig:twolines_ok}.
\label{fig:twolines_bad}}
\end{figure}

In order to study the approximation quality of
Algorithm~\ref{alg:global}, we perturb $C$ by some noise matrix
$N_\delta$ for a series of increasing noise levels~$\delta \in
[10^{-16},1]$. We consider a random matrix $N \in \C^{200,100}$
(generated by MATLAB's {\ttfamily randn} function for its real and
imaginary parts) and set
\begin{equation}
\label{eqn:points_noise}
    A_\delta \defby C + N_\delta,\quad \mbox{where}\quad
    N_\delta \defby \delta * \big(N \odot \abs{N^{[-1]}} \odot \abs{C}\big).
\end{equation}
Thus, the relative perturbation of $C$ by $N_\delta$ in each component 
is exactly $\delta$ (compare~\eqref{eqn:err_condition}). We apply
Algorithm~\ref{alg:global} to each such matrix $A_\delta$, and we denote 
the output by $\hat{z} \defby \twobyone{\hat{s}}{\hat{t}} \defby U^+
\flatten(A_\delta^{[-T]})$.

Figure~\ref{fig:twolines_ok} (top) shows, for each noise level
$\delta$, the relative approximation error $\frac{\norm{A_\delta -
C(\hat{s}, \hat{t})}_F}{\norm{A_\delta}_F}$.  We also plot the
bound~\eqref{eqn:beta_bound} and the relative error
$\frac{\norm{A_\delta - C(s, t)}_F}{\norm{A_\delta}_F}$ made by the
original Cauchy points. We observe that the output $\hat{s}, \hat{t}$
of Algorithm~\ref{alg:global} yields an approximation of the given
data matrix $A_\delta$ by a Cauchy matrix $C(\hat{s}, \hat{t})$ with
approximation error linear in the noise level, and that this
approximation quality is on par with the original Cauchy points.
Moreover, the bound~\eqref{eqn:beta_bound} matches the true residual
rather well.

The computed Cauchy points $\hat{s}, \hat{t}$ are, however, different
from the original ones.  Figure~\ref{fig:twolines_ok} (bottom) shows
the relative recovery error $\norm{z - \hat{z}}_2 / \norm{z}_2$, where
$z \defby \twobyone{s}{t}$.  As for the data approximation error, the
recovery error behaves linearly in the noise level.

We now study the effect of increasing the range of magnitudes in the
coefficients of the Cauchy matrix $C(s,t)$ by setting the imaginary
parts of the vectors $s$ and $t$ to $10^{-6}i$ and $-10^{-6}i$
(instead of $i$ and $-i$), respectively, i.e.,
\begin{equation}
\label{eqn:points_bad}
    s = \text{{\ttfamily linspace}}(-1,1,200) + 10^{-6} i, \quad 
    t = \text{{\ttfamily linspace}}(-1,1,100) - 10^{-6} i.
\end{equation}

Figure~\ref{fig:twolines_bad} (top) shows that this change leads to an
increase of the approximation error $\norm{A_\delta - C(\hat{s},
\hat{t})}_F / \norm{A_\delta}_F$ by about six orders of magnitude,
while the global approximation error of the linearization behaves
nicely with respect to the noise level; cf.~\eqref{eqn:maxcompres} and
corresponding discussion. On the other hand, the relative error of the
recovered Cauchy points $\norm{z - \hat{z}}_2 / \norm{z}_2$ is largely
unaffected by this change; see Figure~\ref{fig:twolines_bad} (bottom).

Notice also the ``wiggly'' behaviour of the blue and red line in
Figure~\ref{fig:twolines_bad} (top); this is due to roundoff error in
computing the row and column means in Algorithm~\ref{alg:global}.
Using a multiply compensated summation~\cite{Priest1992} would yield a
more stable behaviour (at a $\log(mn)$ factor higher operation count).

\begin{figure}[p]
\begin{center}
    \includegraphics[width=0.7\textwidth]{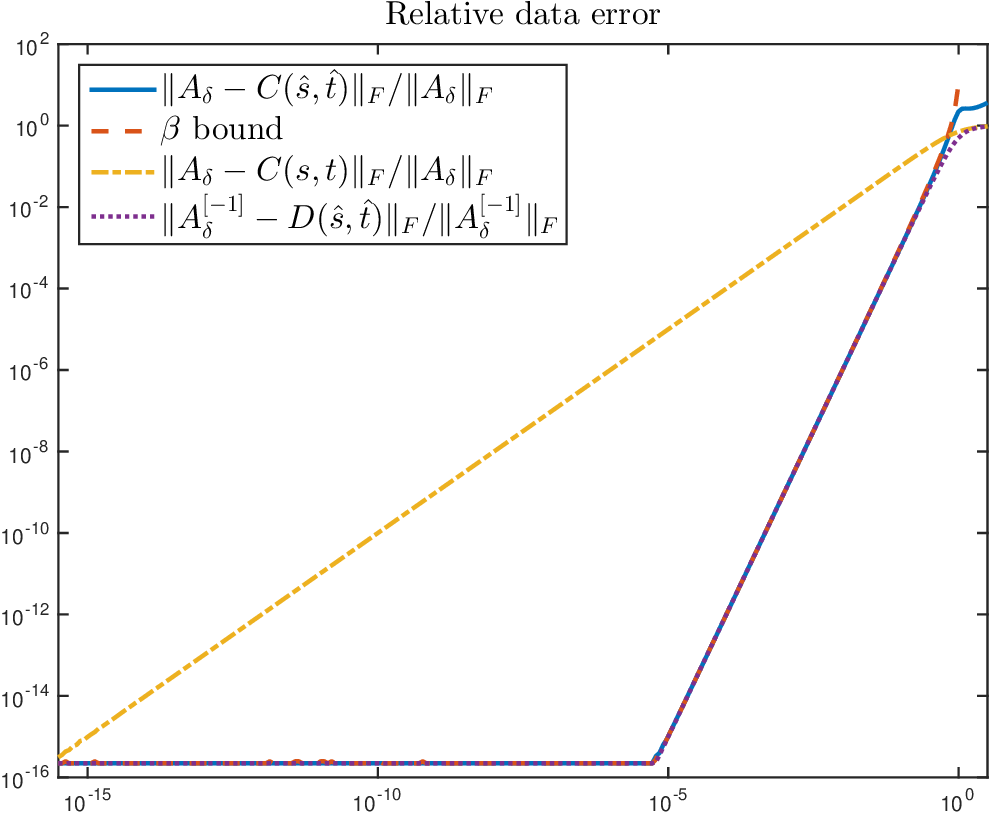}\\
    \includegraphics[width=0.7\textwidth]{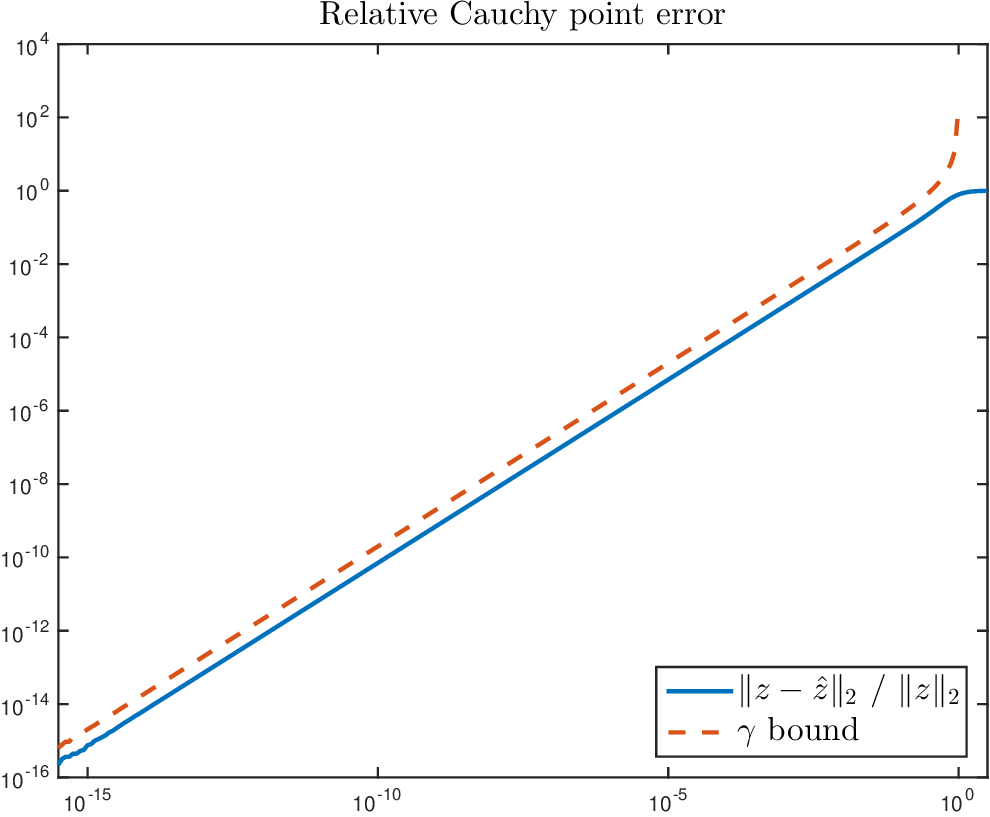}
\end{center}
\caption{Relative data approximation error (top) and Cauchy point
error (bottom) for data
$A_\delta = C([1;-1], [i;-i]) + \delta [1, -1; -1, 1]$. (Values
smaller than the machine precision $\epsilon$ have been set
to $\epsilon$ for cleaner presentation.)
\label{fig:tbtresconst}}
\end{figure}

\begin{figure}[p]
\begin{center}
    \includegraphics[width=0.7\textwidth]{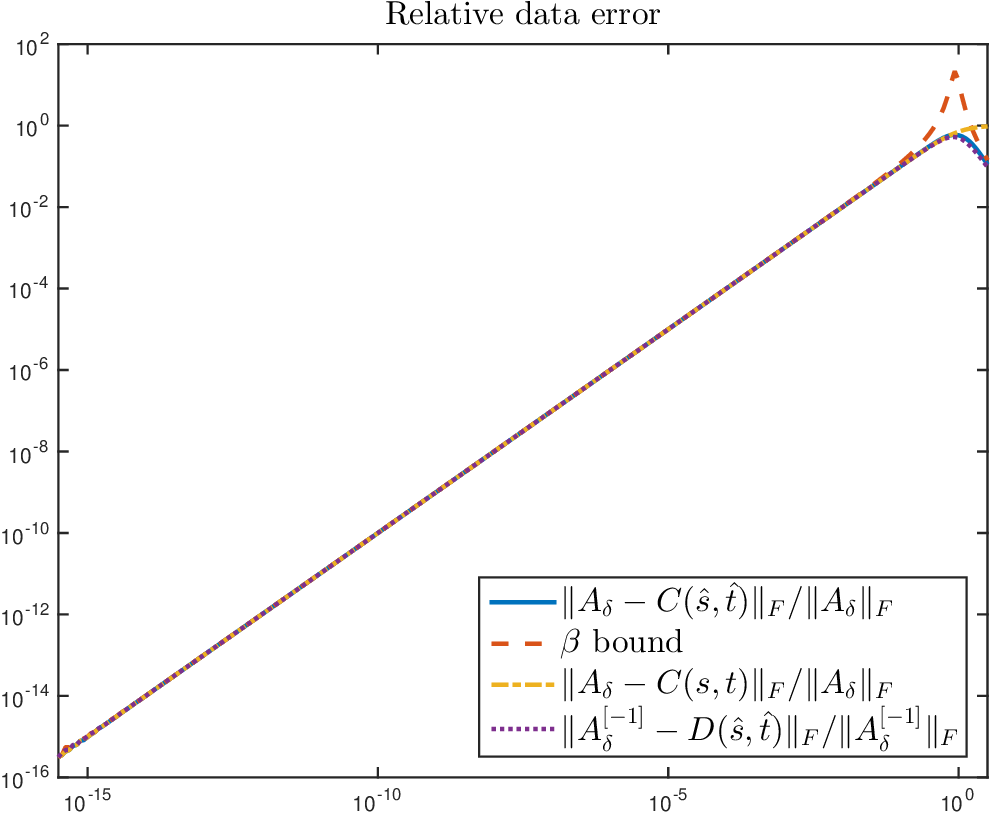}\\
    \includegraphics[width=0.7\textwidth]{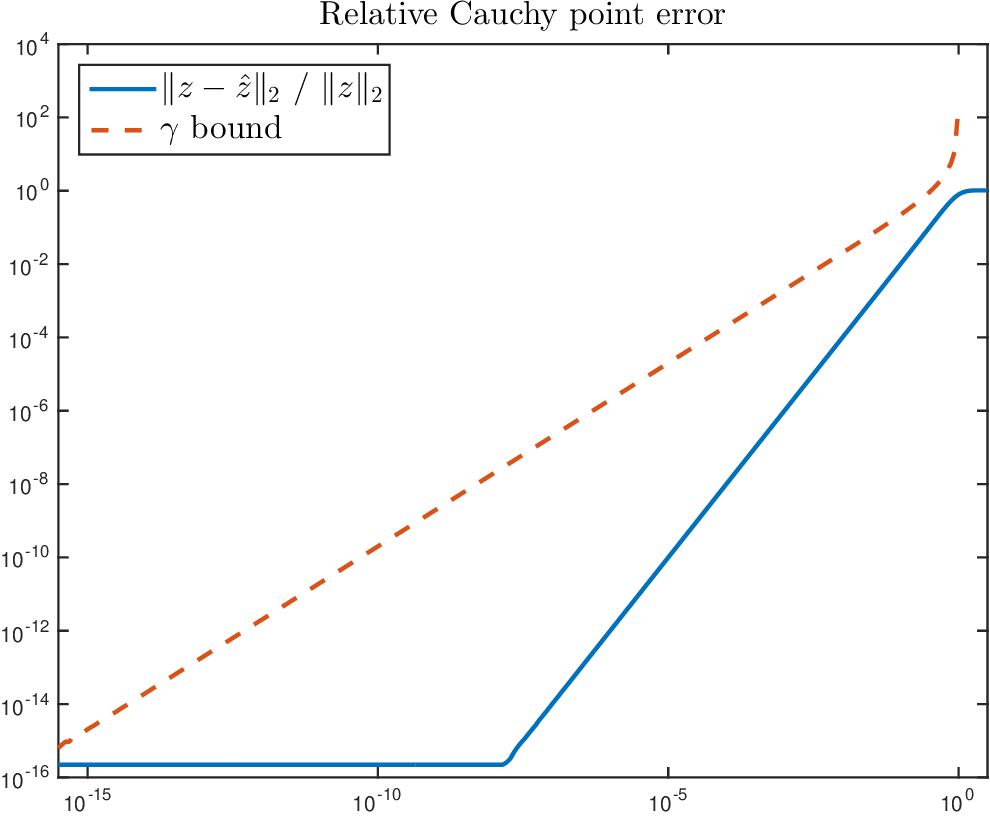}
\end{center}
\caption{Relative data approximation error (top) and Cauchy point
error (bottom) for data
$A_\delta = C([1;-1], [i;-i]) + \delta [-1, -1; -1, -1]$.  (Values
smaller than the machine precision $\epsilon$ have been set
to $\epsilon$ for cleaner presentation.)
\label{fig:tbtpointconst}}
\end{figure}

\subsubsection*{Complementarity of the bounds in Theorems~\ref{thm:data_res}
and~\ref{thm:cauchy_err}}

We consider the $2\times 2$ Cauchy matrix $C(s,t)$ having the (normalized) 
Cauchy points
\begin{equation*}
    s = \begin{bmatrix} 1\\-1 \end{bmatrix}
    \quad \text{and} \quad
    t = \begin{bmatrix} i\\-i \end{bmatrix}.
\end{equation*}

Figure~\ref{fig:tbtresconst} shows the same quantities as in the
previous examples for the matrices
\begin{equation*}
    A_\delta = C(s,t) + \delta
    \begin{bmatrix}
        1  & -1\\
        -1 &  1
    \end{bmatrix},
\end{equation*}
where $\delta$ ranges from $10^{-16}$ to $1.0$. 
While the error made in the recovered Cauchy points increases linearly in $\delta$, 
the data approximation residual $\frac{\norm{A_\delta - C(\hat{s},
\hat{t})}_F}{\norm{A_\delta}_F}$ remains on the machine precision level
until $\delta \approx 10^{-5}$.  A computation shows that for this
particular choice of $s$, $t$ and $N$, the residual of the
linearization corresponding to the solution
$\twobyone{\hat{s}}{\hat{t}} = U^+ \flatten(A_\delta^{[-T]})$ has the
form
\begin{equation*}
    R=A_\delta^{[-1]} - D(\hat{s}, \hat{t}) =
    \begin{bmatrix}
        \frac{4 \delta^3}{1 + 4 \delta^4} & \frac{-4 \delta^3}{1 + 4 \delta^4}\\
        \frac{-4 \delta^3}{1 + 4 \delta^4} & \frac{4 \delta^3}{1 + 4 \delta^4}
    \end{bmatrix},
\end{equation*}
so that $\beta = \norm{A_\delta \odot R}_M$
(see~\eqref{eqn:res_condition}) is smaller than $\epsilon$ 
until $\delta\approx 10^{-5}$.  Consequently, the
bound~\eqref{eqn:beta_bound} implies that the data approximation
residual is about the same size. 

More generally, when for a Cauchy matrix $C(s,t)$ a perturbation $N$ 
is such that  
\begin{equation*}
    \flatten( (C(s,t) + N)^{[-T]} ) \in \image(U)
\end{equation*}
(see~\eqref{eqn:matrixwise}), the data approximation residual will be
zero, while the distance of $\twobyone{\hat{s}}{\hat{t}} = U^+
\flatten(A_\delta^{[-T]})$ to the original Cauchy points
$\twobyone{s}{t}$ can become arbitrarily large.

Using the same Cauchy points as above we now consider a perturbation of the form
\begin{equation*}
    A_\delta = C(s,t) - \delta
    \begin{bmatrix}
        1 & 1\\ 1 & 1
    \end{bmatrix}.
\end{equation*}
The resulting errors are shown in Figure~\ref{fig:tbtpointconst}.  Now
the data approximation error behaves linearly in $\delta$, but the
Cauchy points $s,t$ are exactly recovered up to $\delta \approx
10^{-8}$.  A computation shows that the output of
Algorithm~\ref{alg:global} applied to $A_\delta$ is
\begin{equation*}
    \begin{bmatrix}
        \hat{s} \\ \hat{t}
    \end{bmatrix}
    = U^+ \flatten(A_\delta^{[-T]})
    = \frac{1}{1 + 4 \delta^4} \begin{bmatrix}
    1  - 2 \delta^2  - 2 \delta^3\\
    -1 + 2 \delta^2  - 2 \delta^3\\
    i  + 2i \delta^2 + 2 \delta^3\\
    -i - 2i \delta^2 + 2 \delta^3
    \end{bmatrix},
\end{equation*}
so that, numerically, the recovered Cauchy points are the original
ones until $\delta^2 \approx \epsilon$.  

More generally, for a Cauchy matrix $C(s,t)$ a perturbation
$N$ is such that 
\begin{equation*}
    \flatten( C(s,t)^{[-T]} - (C(s,t) + N)^{[-T]}  )
    \in \image(U)^\perp,
\end{equation*}
then Algorithm~\ref{alg:global} will recover $\twobyone{s}{t}$ exactly,
while the data approximation error can become arbitrarily large.

\section{Concluding remarks}
\label{sec:conclusion}

We presented an efficient algorithm for the approximation of a given
matrix with a Cauchy matrix.  Our approach for solving the
approximation problem is based on the solution of a linear 
least squares problem based on the explicit construction of 
the pseudoinverse of a structured matrix. It would be very interesting 
to investigate whether similar
approximation algorithms can be derived for other displacement
structured matrices like generalized Cauchy matrices or Cauchy-like
matrices; see, e.g.,~\cite{AricoRodriguez2010,Poloni2010,Gu1998}.

\paragraph*{Acknowledgements} The work of R.~Luce was partially
supported by Deutsche Forschungsgemeinschaft, cluster of excellence
``UniCat''. We thank the two anonymous referees for their constructive
comments which helped us to shorten and improve the presentation,
in particular in Theorem~\ref{thm:fast_ls_sol}.

\appendix

\section{An explicit SVD of $U$}
\label{sec:explicit_svd}

\begin{lemma}
Let $m > 1$, set $n \defby m-1$ and $\nu_j \defby \sqrt{1 +
\tfrac{1}{j}}$ for $j=1,\dots,n$. Then the unreduced upper Hessenberg matrix
\begin{equation*}
Q_m \defby
\begin{bmatrix}
    \nu_1 &
        2 \nu_2 &
        \hdots &
        n \nu_n\\
    - \nu_1 &
        2 \nu_2 &
        \hdots &
        n \nu_n\\
    {} &
        - \nu_2 &
        \ddots &
        \vdots\\
    {} &
        {} &
        \ddots &
        n \nu_n\\
    {} &
        {} &
        {} &
        - \nu_n\\
\end{bmatrix}^{\left[-1\right]} \in \R^{m,m-1}
\end{equation*}
satisfies $Q_m^T Q_m = I_{m-1}$ 
and $\ones_m^TQ_m=0$. 
In particular, the columns of $Q_m$ form an orthogonal basis for
the subspace $\{v \in \C^{m} \setcond 1_m^T v = 0 \}$.
\end{lemma}

\begin{proof}
Let $q_i, q_j$ be the $i$th and $j$th column of $Q$, respectively, and
assume without loss of generality that $1 \le i < j \le n$.  In order
to show $Q_m^T Q_m = I_{m-1}$ we compute
\begin{equation*}
q_i^T q_j
= \sum_{k=1}^i \frac{1}{i j \nu_i \nu_j} - \frac{1}{j \nu_i \nu_j}
= \frac{1}{j \nu_i \nu_j} - \frac{1}{j \nu_i \nu_j}
= 0,
\end{equation*}
and for $1 \le j \le n$,
\begin{equation*}
q_j^T q_j
= \sum_{k=1}^j \frac{1}{j^2 \nu_j^2} + \frac{1}{\nu_j^2}
= \frac{1}{j \nu_j^2} + \frac{1}{\nu_j^2}
= \frac{1}{1+j} + \frac{j}{j+1}
= 1.
\end{equation*}
The equation $\ones_m^TQ_m=0$ follows from $\ones_{m}^T q_j =
\frac{j}{j\nu_j} - \frac{1}{\nu_j} = 0$.

\end{proof}

Using the explicitly constructed matrix $Q_m$ in~\eqref{eqn:matrixW}, 
we obtain an explicit matrix of right singular vectors of the matrix $U$ 
in \eqref{eqn:matrixwise}. Orthogonal bases for the eigenspaces of $UU^T$ 
can also be obtained using $Q_m$ as a building block
(cf. the second part of Lemma~\ref{lem:sv_char}), so that a complete 
SVD of $U$ can be explicitly constructed.

From the special structure of $Q_m$ it is not difficult to see that
matrix vector products with $Q_m$ and $Q_m^T$ can be evaluated in
$\bigo{m}$ operations. Consequently, matrix-vector products with the 
SVD-factors of $U$ can be carried out in constant time per vector
component of the output.

\opt{arxiv}{
\bibliographystyle{plain}
}
\opt{els}{
\bibliographystyle{elsarticle-num}
\section*{References}
}

\bibliography{cauchy_approx}

\end{document}